\documentclass[12pt]{article}
\usepackage[english]{babel}
\usepackage{amsmath,amsthm,amsfonts,amssymb,epsfig,hyperref}
\usepackage[left=1in,top=1in,right=1in]{geometry}

\newtheorem{thm}{Theorem}[section]

\newtheorem*{question}{Question}

\title{Recent work on chemical distance in critical percolation}
\author{Michael Damron \thanks{The research of M. D. is supported by NSF grant DMS-0901534.} \\ \small{Georgia Tech} \\ \small{Indiana University, Bloomington}}

\begin{document}

\maketitle

\abstract{In this note, we describe some of the progress recently made on questions regarding the chemical distance in two-dimensional critical percolation by the author, J. Hanson, and P. Sosoe \cite{chem1, chem2}. It is expected that the distance between points in critical percolation clusters scales as $\|\cdot \|^{1+s}$, where $\|\cdot \|$ is the Euclidean distance and $s>0$. First, we review previous work of Aizenman-Burchard and Morrow-Zhang, which together establish a version of $0 < s \leq 1/3$. The main results of our work are in the direction of proving upper bounds on $s$, answering in \cite{chem1} a question from '93 of Kesten-Zhang on the ratio of the length of the shortest crossing of a box to the length of the lowest crossing of a box. The paper \cite{chem2} provides a quantitative version of the result of \cite{chem1}, along with bounds on point-to-point and point-to-set distances.}

\section{Background}

\subsection{The model}

Percolation was introduced by Broadbent and Hammersley and is one of the simplest models displaying critical phenomena. Each edge of the nearest-neighbor lattice $(\mathbb{Z}^d, \mathcal{E}^d)$ is declared either open (occupied) with probability $p \in [0,1]$ or closed (vacant) with with probability $1-p$, and the states of different edges are assumed to be independent. One can think of an open edge as one which allows fluid to flow through it (it is unblocked) and a closed edge as one which does not. The main objects of study are long, system-spanning open connections that occur when $p$ is above a critical threshold. (See the classic text \cite{Grimmett} for an introduction to the subject.)

To be concrete, let $B_n = [-n,n]^d$ be the box of side-length $2n$ centered at the origin. One says that this box has a \emph{left-right open crossing} if there is an open path (sequence of edges $e_1, e_2, \ldots, e_m$ such that each $e_i$ shares an endpoint with $e_{i+1}$ and all $e_i$'s are open) all of whose vertices are in $B_n$ and which connects the left side $\{-n\} \times [-n,n]^{d-1}$ to the right side $\{n\} \times [-n,n]^{d-1}$. Write $H_n$ for the event that there is such a crossing. Then one can show that there is a critical value $p_c = p_c(d)$, which is in the open interval $(0,1)$ for $d \geq 2$, such that
\[
\mathbb{P}_p(H_n) \to \begin{cases}
1 & \text{ if } p > p_c \\
0 & \text{ if } p < p_c
\end{cases}.
\]
Here, $\mathbb{P}_p$ is the measure on configurations of open/closed edges corresponding to parameter $p \in [0,1]$. In fact, the convergence above occurs exponentially quickly. This is not the standard way to define $p_c$, but it is an equivalent way.

The question of what happens for $p=p_c$ is considerably more difficult and, in many cases, unsolved. For $d=2$, it is a well-known result of Kesten that $p_c=1/2$ and, by duality arguments, $\mathbb{P}_{1/2}(H_n) \to 1/2$. For general dimensions $d\geq 3$, it is not known if $\mathbb{P}_{p_c}(H_n)$ is bounded away from 0 or 1, although Aizenman \cite{A97} has shown that a related ``thin brick'' crossing probability is bounded away from 0.

\subsection{Previous results on chemical distance}

The chemical distance $dist_{chem}(x,y)$ between two vertices in a percolation configuration is defined as the length of the shortest open path from $x$ to $y$. If there is no open path, the distance is infinite. To begin with, and to eventually describe the results of \cite{chem1}, we consider the distance from one side of a box to another. So on the event $H_n$ described in the last subsection, we define $S_n$ to be the minimal number of edges in any left-right open crossing of $B_n$.

For $d \geq 2$ and $p>p_c$, the order of $S_n$ can be found using block arguments, leading to $S_n \leq Cn$ with high probability. In fact, one can even derive exponential bounds for upper large deviations on scale $n$. These statements follow from straightforward adaptations of the work of Grimmett-Marstrand \cite{GM90} in '90 and Antal-Pisztora \cite{AP96} in '96. Further and more recent work of Garet-Marchand \cite{GM07} in '07 even allows one to show that $S_n/n$ converges to a constant almost surely.

\[
\text{From now on, we restrict to } d=2 \text{ and } p = p_c = 1/2.
\]
To get a feeling for the behavior of crossings at $p=p_c$, we can imagine moving $p$ from $1$ down to $p_c$. (This requires us to couple all models together with a standard coupling.) For $p$ close to 1, most edges are open, and one can use simple oriented percolation techniques to show $S_n \leq 5n$ with high probability. In particular, we would imagine there are many nearly straight open paths that connect the left and right sides of $B_n$. As $p$ lowers, many of these open edges become closed. At $p_c$, open crossings barely connect both sides, and must avoid so many closed edges that one expects that in a suitable sense
\[
S_n \sim n^{1+s} \text{ for some } s>0.
\]
The main question now is to determine this exponent $s$. As remarked in O. Schramm's ICM paper \cite{S06}, this ``chemical distance exponent'' is believed not to be related to those obtainable by SLE methods, so there are few ideas about how to obtain it rigorously. There are not even many non-rigorous arguments, but numerical results \cite{HS88} suggest that $s \sim .13 \ldots$.

\medskip
\noindent
{\bf Lower bounds on $s$.} The first rigorous result toward bounds on $s$ in '93 was not actually for the shortest crossing, but for the \emph{lowest} crossing. Any self-avoiding path in the box $B_n$ which starts on the left side and ends on the right (and touches each only once) splits the box into two connected components: an upper and a lower. The lowest open crossing in a percolation configuration is defined as the open such crossing whose lower component is minimal. On $H_n$, let $L_n$ be the length of the lowest open crossing.
\begin{thm}[Kesten-Zhang]
There exists $\alpha>0$ such that
\[
\mathbb{P}(0 < L_n \leq n^{1+\alpha}) \to 0.
\]
\end{thm}
In this same paper, the following question was posed. It remained open until last year, and was solved by the author, Hanson, and Sosoe in \cite{chem1}.
\begin{question}[Kesten-Zhang]
Conditional on $H_n$, does $S_n/L_n \to 0$ in probability?
\end{question}

The Kesten-Zhang result does not directly address $S_n$. The first result which did was due to Aizenman-Burchard \cite{AB99} in '99. The following theorem is a direct application of their methods to bound the fractal dimension of systems of random curves. Note that it implies the Kesten-Zhang result.
\begin{thm}[Aizenman-Burchard]
There exists $\beta>0$ such that
\[
\mathbb{P}( 0 < S_n \leq n^{1+\beta} ) \to 0.
\]
\end{thm}
The Aizenman-Burchard proof is remarkable in that it applies to any statistical mechanical models which satisfy some rather weak assumptions. For instance, it was applied again in \cite{ABNW99} to deduce a similar theorem for the minimal spanning tree and the uniform spanning tree. The main idea was to show that with high probability, all open paths in $B_n$ which have diameter of order $n$ must have irregularities on many scales. These irregularities occur in small thin rectangles which contain closed blocking paths. This theorem remains the best lower bound for $s$, providing a version of $s>0$. It was later adapted to near-critical percolation by Pisztora \cite{PUnpublished}, who obtained exponential convergence to zero, and this adaptation was a main tool in the short proof of a quenched version of Kesten's subdiffusive bound for random walk on the incipient infinite cluster (and invasion percolation cluster) by the author, Hanson, and Sosoe in '12 \cite{DHS12}.

\medskip
\noindent
{\bf Upper bounds on s.} The easiest way to give an upper bound on $s$ is to find a left-right open crossing of $B_n$ whose length $L$ one can estimate. Since $S_n$ is the minimal length of all such paths, one has $S_n \leq L$. We have already discussed one candidate path, the lowest crossing, and finding the order of the length of this crossing is what Morrow-Zhang \cite{MZ05} did in '05:
\begin{thm}[Morrow-Zhang]
For each $k \geq 1$, there exists $C_k>0$ such that
\[
C_k^{-1} (n^2 \pi_3(n))^k \leq \mathbb{E}L_n^k \leq C_k (n^2 \pi_3(n)) \text{ for all }n.
\]
\end{thm}
The quantity $\pi_3(n)$ in the statement of their result is the probability of the so-called \emph{(polychromatic) three-arm event}. To describe it, we need to consider the dual lattice, defined as the original square lattice shifted by the vector $(1/2,1/2)$. Given a percolation configuration on $(\mathbb{Z}^2, \mathcal{E}^2)$, we define the dual configuration as follows. Each $e \in \mathcal{E}^2$ has a unique dual edge $e^*$ which bisects it. If $e$ is open in the original percolation configuration, then we set $e^*$ to be open in the dual configuration; otherwise, we set $e^*$ to be closed. The three-arm event to distance $n$ is the event that there are two disjoint open paths (``arms'') connecting $0$ to $\partial [-n,n]^2$ and one closed dual path connecting a dual neighbor of $0$ to $\partial [-n,n]^2$. The quantity $n^2 \pi_3(n)$ is of the same order as the expected number of points in $B_n$ which have three arms to distance $n$.

What is the significance of the three-arm probability? The three-arm event characterizes points on the lowest crossing. Each point $v$ on the lowest crossing $l_n$ has one open arm to the right side of $B_n$ and a disjoint open arm to the left side. Furthermore, no open crossing can use a vertex strictly below $l_n$. In particular, there is no open arc from one side of $v$ to the other below $l_n$. By duality, then, there must be a closed dual path from a dual neighbor of $v$ to the bottom of $B_n$. Thus $v$ has ``three arms.''

The above theorem was actually proved in a related model: site percolation on the triangular lattice, in which vertices (instead of edges) are labeled open or closed. For that model, it is known \cite{SW01} that $\pi_3(n) = n^{-\frac{2}{3} + o(1)}$, so one concludes that $\mathbb{E}L_n \leq C n^{\frac{4}{3} + o(1)}$. Combined with the Aizenman-Burchard result, this establishes a version of
\begin{equation}\label{eq: s_so_far}
 0 < s \leq 1/3,
\end{equation}
which we note is consistent with the numerical value $s \sim .13 \ldots$.

\section{Recent work on the chemical distance}

\subsection{Box crossings}

The next step beyond the inequalities \eqref{eq: s_so_far} is to establish whether the top inequality is strict. This seems to be a difficult question, so here we present partial results. The first is a positive answer from \cite{chem1} to the '93 question of Kesten-Zhang.
\begin{thm}[Damron-Hanson-Sosoe]
Let $d=2$ and $p=1/2$. Conditional on $H_n$, one has $\frac{S_n}{L_n} \to 0$ in probability.
\end{thm}

The method of proof involves the notion of an $\epsilon$-shielded detour around a portion of the lowest crossing $l_n$ of the box $B_n$. An open path $\gamma$ (with vertices $v_0, \ldots, v_m$) in $B_n$ is said to be an $\epsilon$-shielded detour if the following conditions hold:
\begin{enumerate}
\item $v_0, v_m \in l_n$, but $v_1, \ldots, v_{m-1}$ lie in the region strictly above $l_n$,
\item there is a closed dual path connecting a dual vertex adjacent to $v_0$ to one adjacent to $v_m$ in the region above $l_n \cup \gamma$ and 
\item the length of $\gamma$ is at most $\epsilon$ times the length of the portion of $l_n$ from $v_0$ to $v_m$ (the detoured portion of $l_n$).
\end{enumerate}
Given the collection of all $\epsilon$-shielded detours, one chooses a maximal subcollection $\Pi$ of them in the sense that the total length of $l_n$ which they detour is as large as possible. (The existence of shielding paths in item 2 ensures that these detour paths are disjoint.) Then one can show that the union of the paths in $\Pi$ with the portions of $l_n$ not in any of the detoured paths forms an open crossing $\sigma$ of $B_n$. Last, each vertex of $l_n$ which is not detoured by a path from $\Pi$ can be shown not to have any $\epsilon$-shielded detour around it (not only one from $\Pi$). These observations imply
\[
\mathbb{E}[S_n \mid H_n] \leq \mathbb{E}[ \#\sigma \mid H_n] \leq \epsilon \mathbb{E}[L_n \mid H_n] + \sum_{v \in B_n} \mathbb{P}(F_v \mid v \in l_n) \mathbb{P}(v \in L_n),
\]
where $F_v$ is the event that $v$ is not on the detoured path of \emph{any} $\epsilon$-shielded detour. The main difficulty in the proof then is to show that
\[
\lim_n \sup_{v \in B_n} \mathbb{P}(F_v \mid v \in l_n) = 0.
\]
(Actually, this is proved only for $v$ sufficiently far away from the boundary of $B_n$.) Plugging this estimate back into the above shows that
\[
\limsup_n \frac{\mathbb{E}[S_n \mid H_n]}{\mathbb{E}[L_n \mid H_n]} \leq \epsilon + \lim_n \frac{\sup_{v \in B_n} \mathbb{P}(F_v \mid v \in l_n) \mathbb{P}(v \in l_n)}{\mathbb{E}[L_n \mid H_n]} \leq \epsilon.
\]

The above proof gives no hint at the rate of convergence of $S_n/L_n$ to zero. In other words one could not extract an explicit function $a(n) \to 0$ such that $(S_na(n))/L_n \to 0$ in probability. In recent work \cite{chem2}, however, we obtain a quantitative estimate, using a modified ``detour event."
\begin{thm}[Damron-Hanson-Sosoe]
For any $c_1 \in (0,1/4)$, one has
\[
\mathbb{E}S_n \leq \frac{\mathbb{E}L_n}{(\log n)^{c_1}} \text{ for all large }n.
\]
Furthermore,
\[
\mathbb{P}\left( S_n \geq \frac{L_n}{(\log n)^{c_1}} \mid H_n \right) \to 0.
\]
\end{thm}
Of course, one would like to show that $S_n \leq L_n/n^\alpha$ for some $\alpha>0$. One of the main difficulties in this direction is that there is no known way to find a distinguished crossing apart from the highest or lowest crossing.

\subsection{Point-to-set distances}
We finish this note by giving some bounds on chemical distances between points and sets in 2$d$ critical percolation. The first observation is that although point-to-point distances in supercritical percolation are concentrated, those in critical percolation are not. In the statement below, $\{x\leftrightarrow y\}$ is the event that $x$ is connected to $y$ by an open path.
\begin{thm}[Damron-Hanson-Sosoe]
For $e_1 = (1,0)$, one has
\[
\mathbb{E}dist_{chem}(0,e_1)^2 \mathbf{1}_{\{0 \leftrightarrow e_1\}} = \infty.
\]
\end{thm}
\begin{proof}[Sketch of proof]
Let $d$ be the smallest $k$ such that $0$ is connected to $e_1$ in the box $B_{2^k}$. Since these points are not connected in $B_{2^{k-1}}$, there must be a closed dual path separating them in this box. This closed dual path must connect the boundary of the box to the edge dual to $\{0,e_1\}$, back to the boundary. Since $0$ is connected to $e_1$ in $B_{2^k}$, there must be an open path connecting 0 to the boundary of this box, and a disjoint open path connecting $e_1$ to the same boundary. Using standard gluing arguments and a generalization of the FKG inequality, one can thus show that
\[
\mathbb{P}(d = k) \asymp \pi_4(2^{k-1}),
\]
where $\pi_4(2^{k-1})$ is the ``four-arm'' event that 0 and $e_1$ have two disjoint open paths to $\partial B_{2^{k-1}}$ and $\{0,e_1\}^*$ has two disjoint closed dual paths to the same distance. (The open and closed arms alternate.) By bounding this below by a similar ``five-arm'' event, whose exact asymptotic is known, and using Reimer's inequality, one has $\mathbb{P}(d = k) \geq C 2^{-2k+\epsilon}$ for some $\epsilon>0$. However, on $\{d=k\}$, $dist_{chem}(0,e_1) \geq C2^k$. Thus the expectation is bounded below by
\[
C \sum_k 2^{2k} \pi_4(2^{k-1}) \geq C \sum_k 2^{2k} 2^{-2k + \epsilon} = \infty.
\]
\end{proof}

Last, we state a point-to-point and point-to-set asymptotic.
\begin{thm}
There exists $C>0$ such that for all $n$,
\[
\mathbb{E}[dist_{chem}(0,\partial B_n) \mid A_n] \leq C n^2 \pi_3(n),
\]
where $A_n$ is the event that $0$ is connected to $\partial B_n$. Furthermore, there exists $c>0$ such that for all $x,y \in \mathbb{Z}^2$ and $\lambda>0$,
\[
\mathbb{P}(dist_{chem}(x,y) > \lambda |x-y|^2 \pi_3(|x-y|) \mid x \leftrightarrow y) \leq \lambda^{-c}.
\]
\end{thm}
The above results say that the chemical distance between points is unlikely to be larger than $n^2\pi_3(n)$, which is the order of the expected length of the lowest crossing of $B_n$. Because $S_n \leq L_n$ (recall this notation from last section), similar bounds would be consequences of Morrow-Zhang if they were about the shortest crossing of a box. In the point-to-point or point-to-set case, there is no simple replacement for the lowest crossing of a box. So, for example, the proof of the first statement proceeds by bounding the chemical distance above (on the event $A_n$) by the length of some other distinguished path, each of whose vertices still has three arms. If there is, in addition to the open path, a closed dual path $c$ connecting a neighbor of the origin to $\partial B_n$, one can choose for the distinguished path the open path closest to the ``left'' side of $c$. Each point $v$ on this path must have an open arm to 0 and a disjoint open arm to $\partial B_n$. Furthermore, by extremality and duality, there must be a separate closed dual path from a dual neighbor of $v$ to $c$. Following the union of these closed dual paths from $v$ to the boundary $\partial B_n$, one obtains a ``closed arm'' from $v$. In the case that there is no closed dual path from a dual neighbor of $0$ to $\partial B_n$, there must be open circuits around 0 in $B_n$, and one links these circuits together using smaller scale extremal paths. Points on this constructed path do not have three arms to the required distance, so more arm-counting is needed.

\medskip
\noindent
{\bf Acknowledgement.} I thank J. Hanson for comments on a previous draft.

\end{document}